\documentclass[conference, twocolumn, 10pt]{IEEEtran}
\usepackage{amsfonts,amsmath,amssymb,amsthm}
\usepackage{algorithmic}
\usepackage{algorithm}
\usepackage{balance}
\usepackage{times}
\usepackage{bm}
\usepackage{graphicx}
\usepackage{setspace}
\usepackage{comment}
\usepackage{cases}
\graphicspath{{./fig_eps/}}
\usepackage{here}
\usepackage{ascmac} 
\usepackage{array}
\usepackage{float}
\usepackage{cite}
\usepackage{url}
\usepackage{subfigure}
\newtheorem{mydef}{Definition}
\newtheorem{mylem}{Lemma}
\newtheorem{mythm}{Theorem}

\newtheorem{myrem}{Remark}

\newcommand{\rfig}[1]{Fig.\,\ref{#1}} 
 
\newcommand{\req}[1]{\eqref{#1}} 
 
\newcommand{\rtab}[1]{Table\,\ref{#1}}
\newcommand{\rlem}[1]{Lemma\,\ref{#1}}
\newcommand{\rrem}[1]{Remark\,\ref{#1}}
\newcommand{\rsec}[1]{Section\,\ref{#1}}
\newcommand{\rdef}[1]{Definition\,\ref{#1}}

\newcommand{\rthm}[1]{Theorem\,\ref{#1}}

\newcommand{\qedwhite}{\hfill \ensuremath{\Box}}
\begin{document}
\title{\LARGE \textbf{Self-triggered control for constrained systems: \\ a contractive set-based approach}}
\author{Kazumune Hashimoto, Shuichi Adachi, and Dimos V. Dimarogonas
\thanks{Kazumune Hashimoto and Shuichi Adachi are with Department of Applied Physics and Physico-Informatics, Keio University, Yokohama, Japan.}
\thanks{Dimos V. Dimarogonas is with the department of Electrical Engineering, KTH Royal Institute of Technology,  Stockholm, Sweden. His work was supported by the Swedish Research Council (VR), the Swedish Foundation for Strategic Research (SSF) and Knut och Alice Wallenberg foundation (KAW).
}}
\maketitle
\begin{abstract}
In this paper, a self-triggered control scheme for constrained discrete-time control systems is presented. The key idea of our approach is to construct a transition system or a graph structure from a collection of polyhedral sets, which are generated based on the notion of {set-invariance theory}. 
The inter-event time steps are then determined through a standard graph search algorithm to obtain the minimal total cost to a terminal state. The proposed strategy is illustrated through a numerical example. 
\end{abstract}
\section{Introduction}
Efficient network utilization and energy-aware communication protocols between sensors, actuators and controllers have been recent challenges in the community of Networked Control Systems (NCSs). 
To deal with such problems, event and self-triggered control schemes have been proposed as alternative approaches to the typical time-triggered controllers, see e.g.,  \cite{heemels2012a,dimos2010a,heemels2011a,lemmon2009a}. In contrast to the time-triggered case where the control signals are executed periodically, event and self-triggered strategies determine the executions based on the violation of prescribed control performances, such as Input-to-State Stability (ISS) \cite{dimos2010a} and ${\cal L}_\infty$ gain stability \cite{heemels2011a}. 

In particular, we are interested in designing a self-triggered strategy for \textit{constrained} control systems, 
where certain constraints such as physical limitations and actuator saturations need to be explicitly taken into account. 
One of the most popular control schemes to deal with such constraints is Model Predictive Control (MPC) \cite{Mayne2000a}. 
In the MPC strategy, the current control action is determined by solving a constrained optimal control problem online, based on the knowledge of current state information and dynamics of the plant. Moreover, applications of the event-triggered control to MPC have been recently proposed to reduce the frequency of optimization problems, see e.g., \cite{evmpc_linear1,evmpc_linear2,hashimoto2015a,hashimoto2017a}.

In this paper, we present a novel self-triggered control scheme for constrained systems from an alternative perspective to MPC, namely, a perspective from \textit{set invariance theory} \cite{blanchini1999a}. 
Set invariance theory has been extensively studied for the past two decades, see e.g.,  \cite{blanchini1994a,bitsoris1988a,gutman1986a,gilbert1991a}, and it provides a fundamental tool to design controllers for constrained control systems. Two established concepts are those of, a \textit{controlled invariant set} and $\lambda$-\textit{contractive set}. While a controlled invariant set says that the state stays inside the set for all the time, a $\lambda$-contractive set guarantees the more restrictive condition that the state is stabilized to the origin. 
Several different algorithms to compute the controlled invariant and $\lambda$-contractive set have been proposed, e.g., in \cite{blanchini1994a,bitsoris1988a,dorea1999a,hovd2014a}.

Aside from MPC, only a few works have been reported for designing self-triggered strategies for constrained control systems, 
see e.g., \cite{wu2014a,lehmann2012a,seuret2013a}, where the authors focus on the constraint of actuator saturation problem. In \cite{wu2014a}, an ellipsoidal contractive set is obtained under actuator saturation, and the corresponding stabilizing controller is designed by solving Linear Matrix Inequalities (LMIs). Relavant works have been also investigated in \cite{lehmann2012a}, where the authors have proposed event-triggered controllers by incorperating anti-windup mechanisms. 

The self-triggered strategy proposed in this paper takes a different problem formulation and provides a new approach with respect to previous works in the literature. In contrast to earlier results, the self-triggered scheme is provided for a more general class of constrained control systems, where both state and input constraints (including actuator saturations) are taken into account. The proposed approach mainly consists of the two steps; first, based on an assigned $\lambda$-contractive set, a domain of attraction is enlarged and a collection of polyhedral sets is generated through vertex operations. Based on the generated sets, the second step is to  translate them into the corresponding {transition system}, which consists of symbolic states and transitions to represent the original system's behavior. 
By this translation, inter-event time steps can be efficiently found by implementing standard graph search algorithms to obtain the minimal total cost to a terminal state.

This paper is organized as follows. In Section~II, the system description and some preliminaries of invariant set theory are given. In Section~III, several offline procedures to design the self-triggered strategy are given. 
In Section~IV, the proposed self-triggered strategy is presented. 
In Section~V, an illustrative example is given. We finally conclude in Section~VI. \\

\noindent
\textit{(Nomenclature)}: Let $\mathbb{R}_+$, $\mathbb{N}$, $\mathbb{N}_+$ be the positive real, non-negative and positive integers, respectively. The interior of the set ${\cal S} \subset \mathbb{R}^n$ is denoted as $int \{ {\cal S} \}$. A set ${\cal S} \subset \mathbb{R}^n$ is called ${\cal C}$-set if it is compact, convex, and $0 \in int \{{\cal S}\}$. 
For vectors $v_1, \cdots, v_N$, $conv \{ v_1, \cdots, v_N \}$ denotes their convex hull. A set of vectors $\{v_1, \cdots, v_N \}$ whose convex hull gives a set ${\cal P}$ (i.e., ${\cal P} =conv \{ v_1, \cdots, v_N \}$), and each $v_n$, $n\in \{1, 2, \cdots, N\}$ is not contained in the convex hull of 
$v_1, \cdots, v_{n-1}, v_{n+1}, \cdots, v_{N}$ is called a set of \textit{vertices} of ${\cal P}$. For a given $\lambda \in \mathbb{R}$ and a ${\cal C}$-set ${\cal S}\subset \mathbb{R}^n$, denote $\lambda {\cal S}$ as $\lambda {\cal S} = \{ \lambda x \in \mathbb{R}^n : x \in {\cal S}\}$. Given a ${\cal C}$-set ${\cal S} \subset \mathbb{R}^n$, the function $\Psi_{\cal S} : \mathbb{R}^n \rightarrow \mathbb{R}_+$ with $\Psi_{{\cal S}}(x) = {\rm inf} \{\mu : x \in \mu {\cal S}, \mu \geq 0 \}$ is called a \textit{gauge function}. 

\section{System description, strategies and some preliminaries}\label{strategy_sec}
In this section, the system description and an overview of the self-triggered strategy, as well as some established results of set-invariance theory are provided.  
\subsection{System description and control strategy}\label{sys_desc_sec}
Consider a linear time invariant (LTI) system in the discrete time domain given by
\begin{equation}\label{sys}
x ({k+1}) = A x (k) + Bu (k)
\end{equation}
for $k\in \mathbb{N}$, where $x (k) \in \mathbb{R}^n$ is the state and $u (k) \in \mathbb{R}^m$ is the control variable. 
We assume that the pair $(A, B)$ is controllable, and the state and control input are constrained as $x(k) \in {\cal X},\ u(k) \in {\cal U}$, $\forall k \in \mathbb{N}$, where ${\cal X},\ {\cal U}$ are both polyhedral ${\cal C}$-sets described as 
\begin{equation}\label{constraint}
\begin{array}{lll}
{\cal X} = \{x\in \mathbb{R}^n: H_x x\leq h_x \}, \\
{\cal U} = \{u \in \mathbb{R}^m :H_u u\leq h_u \}, 
\end{array}
\end{equation}
where $H_x \in \mathbb{R}^{n_x \times n}$, $H_u \in \mathbb{R}^{n_u \times m}$ and $h_x$, $h_u$ are appropriately sized vectors having positive components. 
The control objective is to steer the state to the origin, i.e., $x(k) \rightarrow 0$ as $k\rightarrow \infty$. To achieve this, the following two-stage controllers will be designed; the first controller is to stabilize the state toward a prescribed, small target set around the origin. The target set will be constructed in the next section, such that a stabilizing controller exists inside the set. While steering the state to the target set, the following self-triggered controller will be active:
\begin{equation}\label{controller}
u (k) = \kappa (x(k_m)) \in {\cal U}, \ \ k \in [k_m , k_{m+1} ), 
\end{equation}
where $\kappa : {\cal X} \rightarrow {\cal U}$ denotes the state-feedback control law, and $k_0, k_1, k_2 \cdots$ with $k_0 = 0$, are the control update times defined as
\begin{equation}\label{transmission_times}
k_{m +1} = k_m + \Gamma (x(k_m)),
\end{equation}
with a state-dependent sampling function $\Gamma : {\cal X} \rightarrow \{ 1, 2, \cdots, j_{\max} \}$. Here the  maximal inter-event time step $j _{\max} \in \mathbb{N}_+$ is set a-priori in order to formulate the self-triggered strategy. \req{controller} implies that the control input is constant between two consecutive inter-event times. 
In designing the first controller, both the control performance and the number of control updates will be evaluated to determine the inter-event time steps. 

Once the state enters the target set, the second control strategy is implemented to achieve our control objective. In contrast to the first strategy, a periodic controller will be designed such that stability of the origin is achieved with maximal possible inter-event time step. 
When designing the second stage, we will not focus on improving the control performance; considering the fact that the state is inside the target set, which is already around the origin, the largest possible inter-event time step is provided to minimize the number of control updates. 

\subsection{Set-invariance theory}\label{set_invariance_sec}
In the following, we define the standard notions of \textit{controlled invariant set} and $\lambda$\textit{-contractive set}, which are  important concepts to characterize invariance and convergence properties for constrained control systems.
\begin{mydef}[Controlled invariant set]\label{lambda_contractive}
A set ${\cal P} \subseteq {\cal X}$ is said to be a controlled invariant set if and only if there exists a control law $g(x)\in {\cal U}$ such that $Ax + Bg(x) \in {\cal P}$ for all $x\in {\cal P}$. 
\end{mydef}
\begin{mydef}[$\lambda$-contractive set]\label{lambda_contractive}
A set ${\cal P} \subseteq {\cal X}$ is said to be a $\lambda$-contractive set for $\lambda\in [0, 1] $, if and only if there exists a control law $g(x)\in {\cal U}$ such that $A x + B g(x) \in \lambda {\cal P}$ for all $x\in {\cal P}$.
\end{mydef}
Roughly speaking, a set ${\cal P}$ is called $\lambda$-contractive set if all states in ${\cal P}$ can be driven into a tighter or equivalent region $\lambda {\cal P}$ by applying a one-step controller. From the definition, a $\lambda$-contractive set ${\cal P}$ is equivalent to a controlled invariant set for $\lambda=1$. 

We review several established results for obtaining a contractive set and the corresponding properties. For given $\lambda \in [0, 1)$ and ${\cal C}$-set ${\cal S}\subset {\cal X}$, there are several ways to efficiently construct a $\lambda$-contractive set in ${\cal S}$. 
Let ${\cal Q}_{\lambda}: \mathbb{R}^n \rightarrow \mathbb{R}^n$ be the mapping 
\begin{equation}\label{kstep_contrl}
{\cal Q}_{\lambda} ({\cal D}) = \{ x\in {\cal X} : \exists u \in {\cal U},\  A x + B u \in \lambda {\cal D} \}
\end{equation}
A simple algorithm to obtain a $\lambda$-contractive set in ${\cal S}$ is to compute $\Omega_j$, $j \in \mathbb{N}$ as
\begin{equation}\label{iterative_procedure}
\Omega_0 = {\cal S},\ \ \ \Omega_{j+1} = {\cal Q} _{\lambda } (\Omega_{j})  \cap {\cal S}, 
\end{equation}
and then it holds that the set ${\cal P} = \lim_{j\rightarrow \infty} \Omega_j$ is $\lambda$-contractive, see e.g., \cite{blanchini1994a}. Since ${\cal S}$ is ${\cal C}$-set, it is also shown in \cite{blanchini1994a,darup2017a}, that the set $\Omega_j$ converges in the sense that for every $\varepsilon > 0$, there exist a finite $j \in \mathbb{N}_+$ such that ${\cal P}\subseteq  \Omega_{j} \subseteq (1+\varepsilon) {\cal P}$.
Several other algorithms have been recently proposed, see e.g., \cite{hovd2014a} and also \cite{gutman1986a,darup2017a} for a detailed convergence analysis.
The following lemma illustrates the existence of a (non-quadratic) Lyapunov function in a given $\lambda$-contractive set:
\begin{mylem}\label{stability_lem}{\normalfont \cite{blanchini1994a,darup2015a}:}
Let ${\cal P} \subset {\cal X}$ be a $\lambda$-contractive ${\cal C}$-set with $\lambda \in [0, 1]$ and the associated gauge function $\Psi_{\cal P} : {\cal P} \rightarrow \mathbb{R}_+$. Then, there exists a control law $g : {\cal X} \rightarrow {\cal U}$ such that 
\begin{equation}\label{set_induced_lyapunov}
\Psi_{\cal P} ( Ax + B g (x)) \leq \lambda \Psi_{\cal P} ( x),
\end{equation}
for all $x \in {\cal P}$.
\end{mylem}
\rlem{stability_lem} follows directly from \rdef{lambda_contractive}.
If $\lambda < 1$, \req{set_induced_lyapunov} implies the existence of a stabilizing controller in ${\cal P}$ in the sense that the output of the gauge function $\Psi_{\cal P}$ is guaranteed to decrease with the constant rate $\lambda$. The gauge function $\Psi_{\cal P}$ defined in ${\cal P}$ is known as \textit{set-induced Lyapunov function} in the literature; for a detailed discussion, see e.g., \cite{blanchini1994a}. 

If a given polyhedral ${\cal C}$-set ${\cal P}= conv\{ v_1, v_2, \cdots, v_{n_p} \}$ is $\lambda$-contractive, then from \rdef{lambda_contractive}, there exist a set of controllers $u_n \in {\cal U}$, $n\in \{ 1, 2, \cdots, n_p\}$ such that $A v_n + B u_n \in \lambda {\cal P}$ for all $n\in \{ 1, 2, \cdots, n_p\}$. Then, the following scaling property holds for ${\cal P}$:
\begin{mylem}\label{lem1}{\normalfont \cite{blanchini1994a}:}
Suppose that there exists $\gamma > 1$, such that
\begin{equation}\label{scaled_constraint}
\gamma u_n \in {\cal U},\ \ \gamma v_n \in {\cal X},\ \forall n \in \{ 1, \cdots, n_p\}.
\end{equation}
Then, a scaled set $\gamma {\cal P} = \{ \gamma x : x \in {\cal P} \}$ is also $\lambda$-contractive.
\end{mylem}
\rlem{lem1} states that for a given polyhedral, $\lambda$-contractive ${\cal C}$-set, the corresponding larger, scaled set is also shown to be $\lambda$-contractive as long as the constraint satisfactions \req{scaled_constraint} are fulfilled. Together with \rlem{stability_lem},  \rlem{lem1} illustrates the possibility to enlarge the domain of attraction to guarantee stability of the origin.

\section{Offline procedure to design the self-triggered strategy}
In this section an offline procedure to design the self-triggered strategy is provided. 
The procedure is mainly devided by the two steps;  
first, the target set is assigned to be $\lambda$-contractive in a given polyhedral set, and a domain of attraction is iteratively enlarged. In this step, a collection of polyhedral sets are generated under different control step sizes. 
Based on the generated sets, the second step is to translate them into the transition system, which consists of symbolic states and transitions to represent the control system's behavior \req{sys}. This translation allows us to utilize well-known graph search algorithms, so that the shortest path to a terminal state can be efficiently found. Finding such path is a key idea to determine inter-event time steps in the online self-triggered implementation. 
\subsection{Target set assignment}\label{targ_assign_sec}
Let us first assign a {target set}, denoted as ${\cal P}_0\subset {\cal X}$, to which state trajectories need to enter in finite time. To do this, consider an arbitrarily small, polyhedral ${\cal C}$-set ${\cal S}_0 \subset {\cal X}$ around the origin. Motivated by \rlem{stability_lem} that a stabilizing controller exists inside a contractive set, and by \rlem{lem1} that the domain of attraction can be enlarged, we assume that the target set is assigned to be $\lambda$\textit{-contractive} ($\lambda <1$) in ${\cal S}_0$. 
Note that for given $\lambda \in [0, 1)$ and the polyhedral ${\cal C}$-set ${\cal S}_0$, one can always construct the desired $\lambda$-contractive set ${\cal P}_0 \subseteq {\cal S}_0$ by polyhedral operations to compute \req{iterative_procedure}, or by other procedures, e.g., \cite{hovd2014a}. The obtained target set ${\cal P}_0$ can be denoted as 
\begin{equation}
{\cal P}_0 = conv\{v_1, v_2, \cdots, v_N \}, 
\end{equation}
where $v_n, n\in \{1, 2, \cdots, N \}$ represent the vertices of ${\cal P}_0$, and $N$ represents the number of them. 

\subsection{Enlarging the domain of attraction}\label{enlarge_sec}
As a scaled set $\gamma {\cal P} _0$ ($\gamma \geq 1$) becomes contractive from \rlem{lem1}, we can possibly enlarge the domain of attraction based on the original ${\cal P}_0$. Thus, in this subsection we provide a procedure to iteratively maximize the domain of attraction by solving an optimization problem on vertices. 

So far, we have only considered through the notion of $\lambda$-contractive set, how much the state gets closer to the origin under a \textit{one-step} controller. However, from a self-triggered view point, it may be useful to analyze how the state can be closer to the origin under a $j$\textit{-step} ($j>1$) constant controller. This motivates us to provide the following algorithm, in which a collection of sets are generated not only with a one-step controller, but also with longer control step sizes: \\

\noindent
\textbf{Algorithm 1}: 
For each $j \in \{1, 2, \cdots, j_{\max} \}$, find the sets ${\cal P}_{j, 0},\ {\cal P}_{j, 1},\ {\cal P}_{j, 2}, \cdots $ by the following steps:
\begin{enumerate}
\item (Initialization): 
Set ${\cal P} _{j,0} = a_{j,0}{\cal P} _0$, with $a_{j,0} = 1$.
\item For given $\ell \in \mathbb{N}$, $a_{j, \ell} \in \mathbb{R}_+$ and ${\cal P}_{j,\ell} = a_{j, \ell} {\cal P} _0 $, solve the following problem: 
\begin{equation}\label{find_a}
\underset{u_1,u_2,\cdots,u_{N}} {\max}\ a, \ \ \ \ a \in \mathbb{R}_+
\end{equation}
subject to $a > a_{j, \ell}$ and
\begin{equation}\label{constu}
a v_n \in {\cal X}, \ \ u_n \in {\cal U},\ \ \forall n \in \{1, 2, \cdots, N\},
\end{equation}
\begin{eqnarray}\label{constx}
A^j a v_n + \sum^{j} _{i=1} A^{i-1} B u_n \in {\cal P}_{j, \ell}, \ \forall n \in \{1, 2, \cdots, N\}. 
\end{eqnarray}
\item Let $a^*$ be a solution to the problem in step (2). 
If the problem does not have a solution, or it satisfies $a_{j, \ell} < a^* < a_{j, \ell} + \bar{a}$ for a given threshold $\bar{a}>0$, then terminate the algorithm. 
Otherwise, set $a_{j, \ell+1} = a^*$, ${\cal P}_{j, \ell+1} = a_{j, \ell+1} {\cal P}_0$,\ \ $\ell \leftarrow \ell+1$, and go back to the step (2). \qedwhite
\end{enumerate}

Recall that $j_{\max}$ is defined in \rsec{sys_desc_sec} as the maximal inter-event time step. 
For each $j\in \{1, \cdots, j_{\max} \}$ in Algorithm~1, 
we try to find a collection of sets by solving the optimization problem given by \req{find_a}, subject to \req{constu} and \req{constx}. In \req{constx}, $a v_n$, $n\in \{1, \cdots, N \}$ represent the vertices of the scaled set $a {\cal P}_0$, and $A^j a v_n + \sum^{j} _{i=1} A^{i-1} B u_n$ represents a point from $a v_n$ by applying a constant controller $u_n \in {\cal U}$ for $j$ steps. Thus, for given $a_{j, \ell} \in \mathbb{R}_+$ and the scaled set ${\cal P} _{j, \ell} = a_{j, \ell} {\cal P} _0$, the problem \req{find_a} aims to find the largest possible scaled set $a {\cal P}_0$ such that it contains ${\cal P}_{j, \ell}$ (i.e., $a > a_{j, \ell}$) and any vertex of $a {\cal P}_0$ can be driven into ${\cal P}_{j, \ell}$ by applying a $j$-step constant controller. 
Note that the optimization problem in the algorithm is a linear programming problem, since both constraints \req{constu} and \req{constx} are linear. 

In step (3) of Algorithm 1, the design parameter $\bar{a} >0$ is set as the threshold to terminate the algorithm; the algorithm terminates when no enlargement can be done or the enlargement goes below the threshold $\bar{a}$. 
\begin{myrem}[On the termination of Algorithm 1]\label{termination}
\normalfont
Note that the algorithm is guaranteed to terminate in a finite number of iterations due to the threshold $\bar{a}$ and the constraint \req{constu}. To verify this, suppose that the algorithm does \textit{not} terminate for some $j$ and is iterated for infinite number of times. Due to the threshold $\bar{a}$ in step (3) in Algorithm 1, we have $a_{j, \ell+1} \geq a_{j, \ell} + \bar{a} \geq a_{j, \ell-1} + 2\bar{a} \cdots \geq a_{j, 0} + (\ell+1) \bar{a}$, which implies that $a_{j, \ell} \rightarrow \infty$ (i.e., ${\cal P}_{j, \ell} \rightarrow \mathbb{R}^n$) as $\ell \rightarrow \infty$. However, this contradicts the state constraint in \req{constu}, which imposes ${\cal P}_{j, \ell} \subseteq {\cal X} \subset \mathbb{R}^n$. Thus, it is shown that the Algorithm~1 terminates in a finite number of iterations. \qedwhite
\end{myrem}

For each $j\in \{ 1, \cdots, j_{\max} \}$, let 
\begin{equation}\label{generated_sets}
{\cal P}_{j, 0},\ {\cal P} _{j, 1},\ \cdots, {\cal P} _{j, \ell_j}, 
\end{equation}
be the collection of sets by applying Algorithm 1, where ${\ell_j}$ denotes the total number of generated sets for each $j$. 
As an output of the algorithm, we also have the positive constants $a_{j, 0}, a_{j, 1}, \cdots, a_{j, \ell_j}$ for each $j\in \{1, \cdots, j_{\max} \}$ with $a_{j, 0}=1$ and ${\cal P} _{j, \ell} = a_{j, \ell} {\cal P}_0$, $\ell \in \{0, \cdots, \ell_j \}$. Thus, each $a_{j, \ell}$ represents the size of ${\cal P} _{j, \ell}$, with respect to ${\cal P}_0$, and it holds that ${\cal P} _{j, 0} \subset {\cal P}_{j, 1} \subset \cdots \subset {\cal P}_{j, \ell_j}$ since $a_{j, 0} < a_{j, 1} < \cdots < a_{j, \ell_j}$. 
\begin{mylem}\label{convergence}
Let ${\cal P}_{j, 0}, \cdots, {\cal P}_{j, \ell_j}$, $j\in\{1, \cdots, j_{\max} \}$ be the generated sets by applying Algorithm~1. Then, for every $x \in {\cal P}_{j, \ell}$ with $j\in \{1, \cdots, j_{\max}\}$, $\ell \in \{ 1, \cdots, \ell_j \}$, there exists $u\in {\cal U}$, such that $x' = A^j x + \sum^{j} _{i=1} A^{i-1} B u \in {\cal P}_{j, \ell-1}$. 
\end{mylem}
The proof is given in the Appendix. 
\rlem{convergence} states that for every $x$ in ${\cal P}_{j, \ell}$, there exists a constant, $j$-step controller $u\in {\cal U}$, such that the state can enter the smaller set ${\cal P}_{j, \ell-1} \subset {\cal P}_{j, \ell}$. 
By following this argument, it is deduced that any state starting from ${\cal P}_{j, \ell}$ will eventually enter the target set ${\cal P}_{j, 0} = {\cal P}_0$ in finite time. 
Let ${\cal P}_{\max} = a_{\max} {\cal P}_0$ be the largest set in \req{generated_sets} given by $a_{\max} ={\rm max}\ a_{j, \ell},\ j\in\{1, \cdots, j_{\max} \},\ \ell \in\{0, \cdots, \ell_j \}$. Since ${\cal P}_{\max}$ is the largest set and any state in ${\cal P}_{\max}$ can enter ${\cal P}_0$, ${\cal P}_{\max}$ is regarded as the maximal domain of attraction obtained by Algorithm~1. \\

\noindent
\textit{(Example 1)}: Consider the double integrator system discretized under $0.1$ sampling time interval; $A = [1.0\ 0.1;\ 0\ 1.0]$, $B = [0.005;\ 0.1]$ and assume that $|u|\leq 2.0$. 
For a given box set ${\cal S}_0 = \{x \in \mathbb{R}^2 : |x_1|\leq0.2, |x_2|\leq 0.2 \}$, the target set ${\cal P} _0\subseteq {\cal S}_0$ with $\lambda =0.96$ is computed and Algorithm~1 is implemented with $j_{\max} = 30$, $\bar{a}=0.01$. The simulation is conducted on Matlab 2016a, using Multi-Parametric Toolbox (MPT3).
The generated sets \req{generated_sets} for the case $j=5, 10, 20$ are illustrated in \rfig{levelset_result}. In this example, the maximal domain of attraction ${\cal P}_{\max}$ is attained with $j = j_{\max}$ and is also given in the figure.  
\qedwhite 

\begin{figure}[t]
   \centering
    \subfigure[$j=5$]
      {\includegraphics[width=4.3cm]{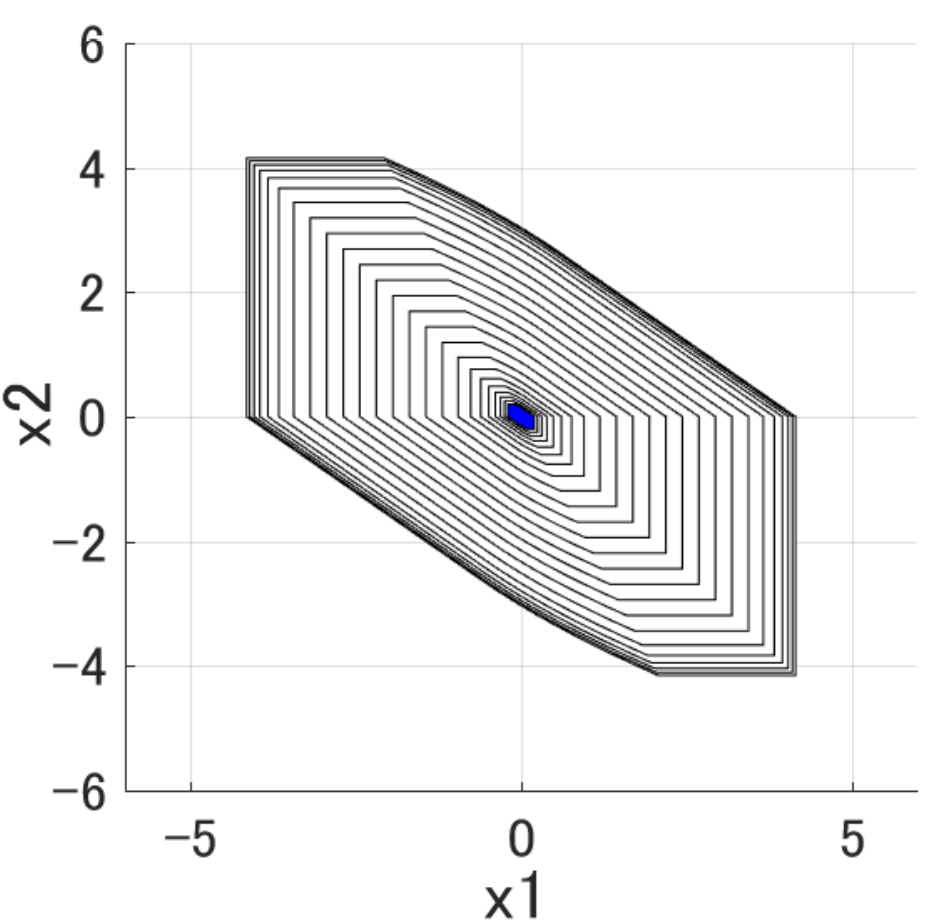}}
      \hspace{1pt}
    \subfigure[$j=10$]
      {\includegraphics[width=4.05cm]{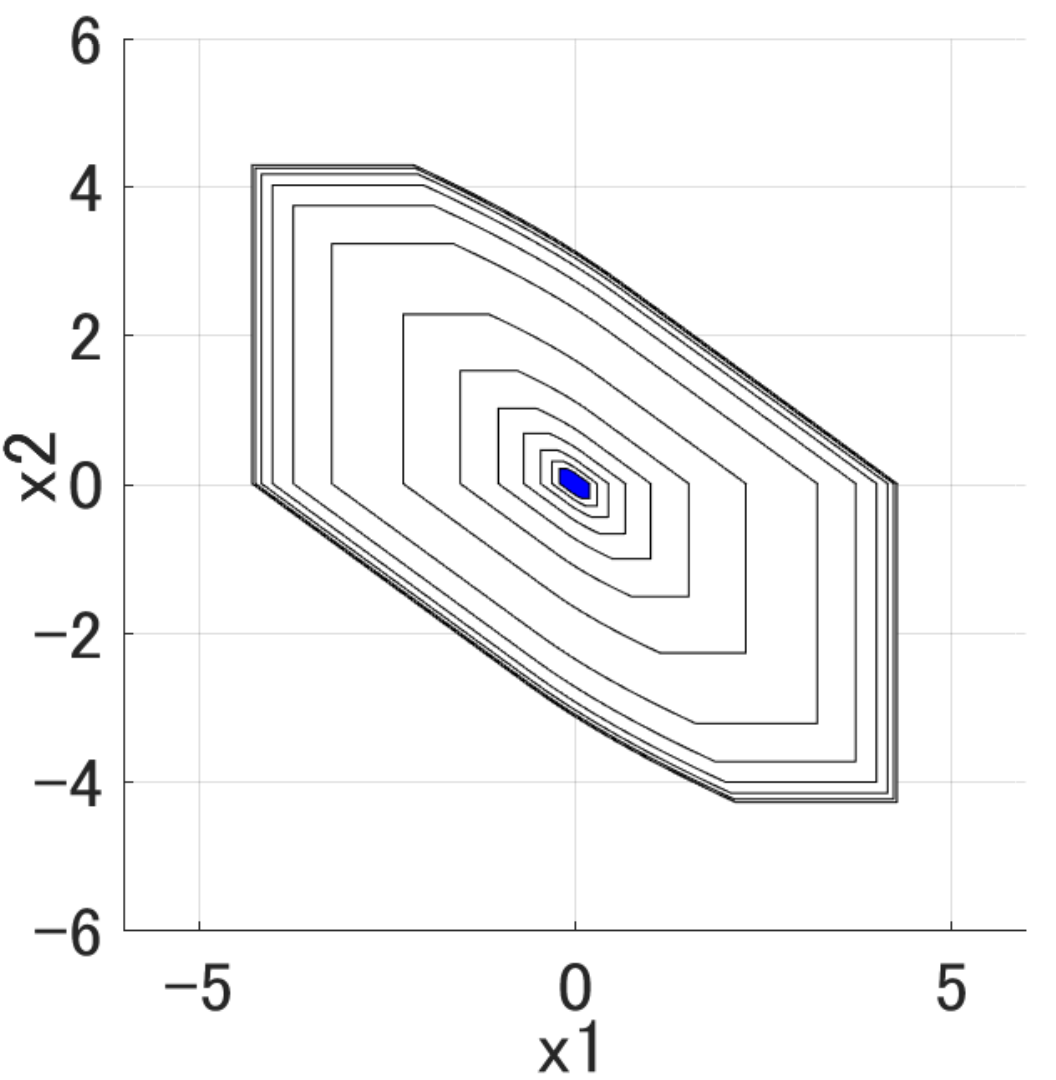}}
    \subfigure[$j=20$]
      {\includegraphics[width=4.3cm]{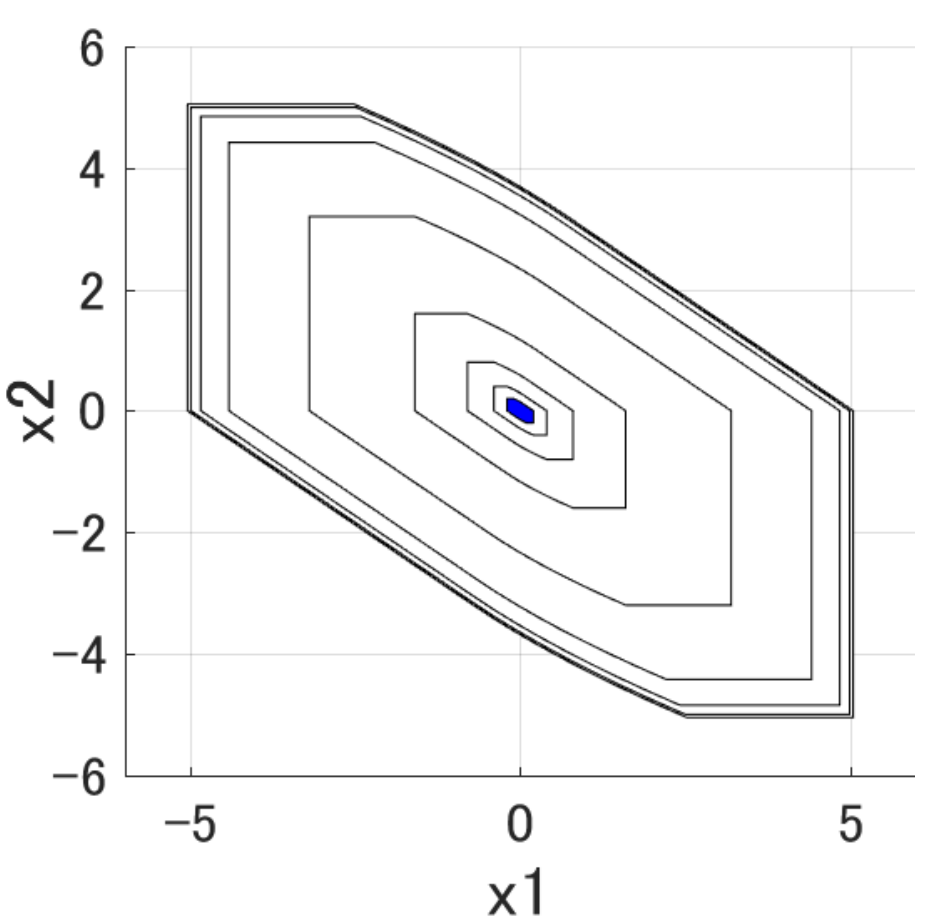}}
    \subfigure[Maximal domain of attraction ${\cal P}_{\max}$ (black solid line).]
      {\includegraphics[width=4.15cm]{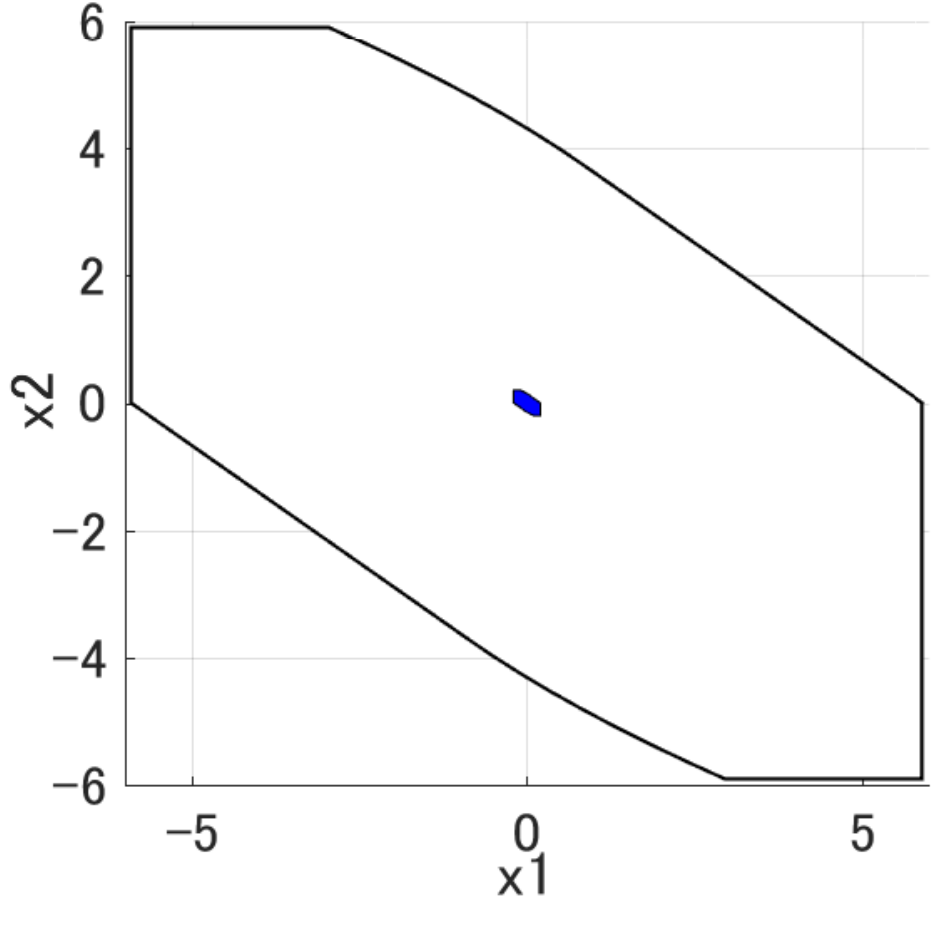}\label{doa}}
    \caption{A collection of generated sets ${\cal P}_{j, 0}, \cdots, {\cal P}_{j, \ell_j}$ according to Algorithm~1 for the case $j=5, 10, 20$, and the maximal domain of attraction ${\cal P}_{\max}$. 
The blue small region represents the original target set ${\cal P} _0$, and the black solid lines in each figure show the boundaries of ${\cal P}_{j, 0}, {\cal P}_{j,1}, \cdots, {\cal P}_{j,\ell_j}$ with ${\cal P}_{j,0} = {\cal P}_0$ and ${\cal P}_{j, 0}\subset {\cal P}_{j,1}\subset \cdots \subset {\cal P}_{j,\ell_j}$.}
    \label{levelset_result}
 \end{figure}

\subsection{Transition systems based on generated sets}\label{transition_system_sec}
Having obtained the sets \req{generated_sets}, we next translate them into the transition system and the corresponding {graph structure}, which consists of symbolic states (nodes) and transitions (edges). 
In the following, we provide a standard definition of transition systems:
\begin{mydef}\label{transition_system}
A transition system is a tuple ${TS} = ({\cal S}, \delta_s, {\cal W}_s, {\cal F}_s) $ where ${\cal S}$ is a set of states, 
$\delta_s \subseteq {\cal S} \times {\cal S}$ is a transition relation, ${\cal W}_s: \delta_s \rightarrow \mathbb{R}_+$ is a cost function associated with each transition, and ${\cal F}_s \subseteq {\cal S}$ is a set of terminal states. 
\end{mydef}
Suppose that $x \in {\cal P} _{j, \ell}$ with $j\in \{1, \cdots, j_{\max} \}$, $\ell \in \{1, \cdots, \ell_{j} \}$. Then, from \rlem{convergence} there always exists a constant controller $u\in {\cal U}$, such that $x' = A^j x + \sum^{j} _{i=1} A^{i-1} B u \in {\cal P}_{j,\ell-1}$. We interpret this fact by newly defining the \textit{symbolic} states as $s_{j, \ell}$, $j\in \{ 1, 2, \cdots, j_{\max}\}$, $\ell \in \{ 0, 1, \cdots, \ell_j \}$, and the corresponding transition $(s_{j, \ell}, s_{j, \ell-1}) \in \delta_j$, where the state $s_{j, \ell}$ represents the set ${\cal P}_{j, \ell}$, and the transition $(s_{j, \ell}, s_{j, \ell-1}) \in \delta_j$ indicates that for every $x \in {\cal P}_{j, \ell}$, $x$ can always enter ${\cal P}_{j, \ell-1}$. The illustration of this interpretation is depicted in \rfig{graph_interpret}. 

Formally, we define the following notion of a \textit{$j$-step symbolic transition system}: 
\begin{mydef}\label{jsymbolic_system}
A $j$-step symbolic transition system for each $j\in \{1, 2, \cdots, j_{\max} \}$, is a tuple $TS_j = ( {\cal S}_j, \delta_{j}, {\cal W}_j, {\cal F}_j)$, where 
\begin{itemize}
\item ${\cal S}_j = \{ s_{j,0},\ s_{j,1},\ \cdots, s_{j,\ell_j} \}$ is a set of symbolic states; 
\item $\delta_{j} \subseteq {\cal S}_j \times  {\cal S}_j$ is a transition relation, where $(s_{j, \ell}, s_{j,\ell-1}) \in \delta_{j}$ for all $\ell \in \{1, \cdots, \ell_j\}$; 
\item ${\cal W}_j: {\delta}_j \rightarrow \mathbb{R}_+$ is a cost function associated with the transition $\delta_j$ and is given by 
\begin{equation}\label{reward}
{\cal W}_j ( \delta_j (s_{j, \ell}, s_{j, \ell-1})) = p j /\left(a_{j, \ell} - a_{j, \ell-1} \right) + q /j, 
\end{equation}
where $p, q \in \mathbb{R}_+$ are given weights satisfying $p, q > 0$. 
\item ${\cal F}_j \subset {\cal S}_j$ is the terminal state given by ${\cal F}_j = \{ s_{j, 0} \}$.
\end{itemize}
\end{mydef}
\begin{figure}[t]
   \centering
    \subfigure[Translating two sets ${\cal P}_{j, \ell}$, ${\cal P}_{j, \ell-1}$ into the corresponding two symbolic states $s_{j, \ell}, s_{j, \ell-1}$.]
      {\includegraphics[width=6.0cm]{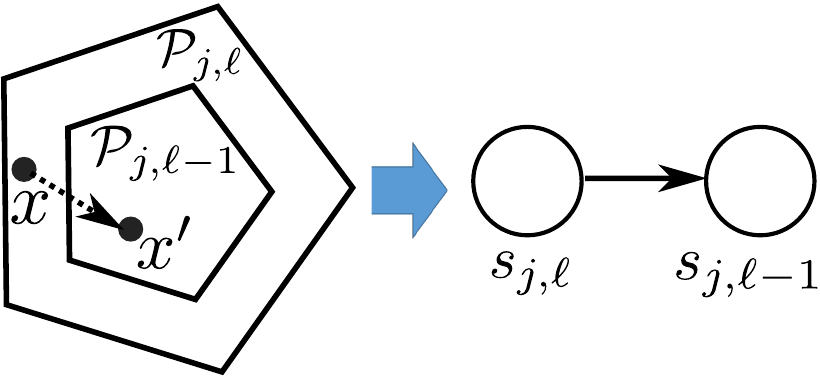}\label{graph_interpret}}
      \hspace{5pt}
    \subfigure[Generated symbolic transition systems $TS_j$.]
      {\includegraphics[width=6cm]{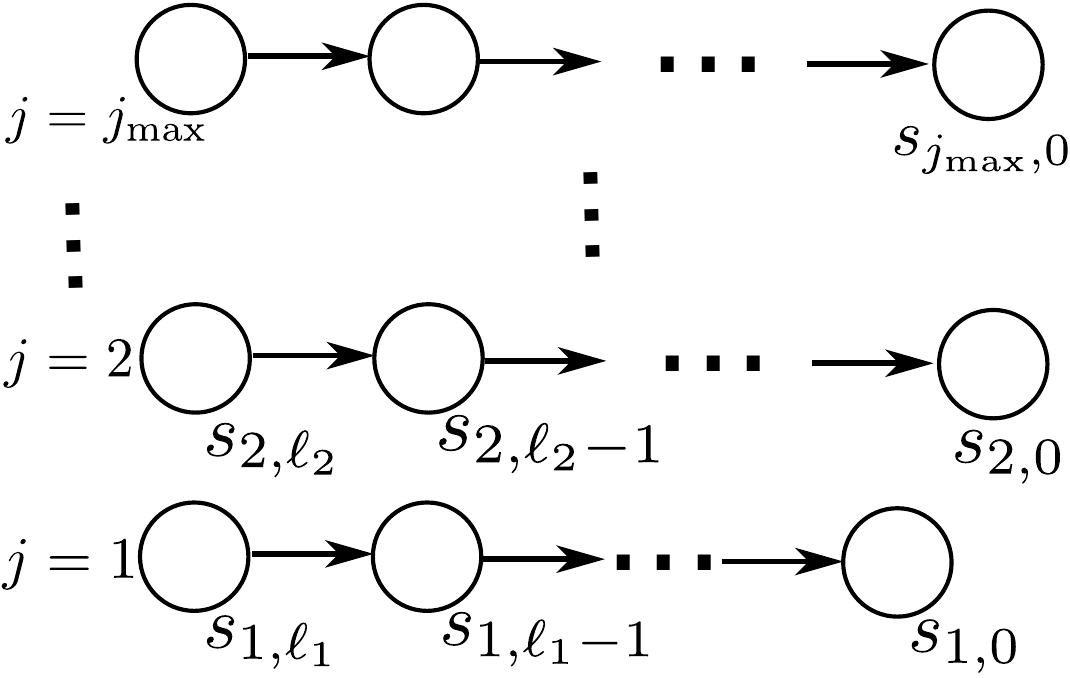}\label{graph_interpret2}}
    \caption{The illustration of translating into the transition system $TS_j$ in \rdef{jsymbolic_system}.}
 \end{figure}

The obtained transition system can be seen as a graph structure, where symbolic states represent nodes and transitions represent edges; see the illustration of $TS_j$ in \rfig{graph_interpret2}. 
It can be seen from \rfig{graph_interpret2}, that each symbolic state $s_{j, \ell}$ can eventually reach the terminal state $s_{j, 0}$. This in turn means in the state domain ${\cal X}$ that any state $x$ in ${\cal P}_{j, \ell}$ can eventually enter the target set ${\cal P} _0$ in finite time. 
Thus, stabilizability of $x$ to the desired target set is equivalent to the reachability of symbolic state to the terminal state. 

The cost function \req{reward} represents a penalty associated to each transition $(s_{j, \ell}, s_{j, \ell-1}) \in \delta_j$, in terms of control performance and the inter-event time steps. Recall that $a_{j, \ell}$ represents the size of ${\cal P}_{j,\ell}$ with respect to ${\cal P}_0$, i.e., ${\cal P}_{j, \ell} = a_{j, \ell} {\cal P}_0$.
The term $(a_{j, \ell} - a_{j, \ell-1})$ indicates how much the set size is \textit{reduced} from ${\cal P}_{j, \ell}$ to ${\cal P}_{j, \ell-1}$ by applying a $j$-step controller, which implies, how much the state $x$ is guaranteed to be closer to the target set ${\cal P} _{0}$. Thus the term $(a_{j, \ell} - a_{j, \ell-1})/j$ represents the rate of convergence to the target set. As achieving larger $(a_{j, \ell} - a_{j, \ell-1})/j$ leads to a better control performance, we take the reciprocal of it to represent as a {cost} function (see \req{reward}). 
From a self-triggered control view point, less control updates will be obtained when control inputs can be applied constantly longer (i.e., when $j$ becomes larger). 
Thus, the second part in \req{reward} involves $1/j$ to represent a cost for the inter-event time steps; as $j$ gets larger, then we obtain a larger inter-event time step and a smaller cost is obtained. 

The values of $p, q >0$ in \req{reward} represent tuning weights associated with each part of the cost. In later sections, we will illustrate through a simulation example that the trade-off between the control performance and inter-event time steps can be regulated by appropriately tuning these parameters. 
\subsection{Composition}\label{composition_sec}
Although each symbolic transition system provides a path to the terminal state, it gives such path only in a single transition system; if the symbolic state \textit{could} jump to another state in another transition system, (e.g., $s_{j, \ell} \rightarrow s_{j' , \ell'}$ with $j\neq j'$), then it could generate other paths that could reduce the total cost to the terminal state. Therefore, in this subsection we provide a few more steps by adding several edges among different transition systems, and then construct a composite model to synthesize the self-triggered strategy. 

For a given ${\cal P}_{j, \ell}$, suppose that there exist $j' (\neq j)$, $\ell' \in \{1, \cdots, \ell_{j'} \}$ such that ${\cal P}_{j', \ell'} \subset {\cal P}_{j, \ell} \subseteq {\cal P}_{j', \ell'+1}$ (i.e., $a_{j', \ell'} < a_{j, \ell} \leq a_{j', \ell'+1}$), see such illustration in \rfig{add_edge}. 
In this case, from \rlem{convergence} every $x \in {\cal P}_{j', \ell'+1}$ can be driven into ${\cal P}_{j',\ell'}$ with a $j'$-step constant controller. 
This means that every $x \in {\cal P}_{j, \ell}$ can also be driven into ${\cal P}_{j', \ell'}$ since ${\cal P}_{j, \ell} \subseteq {\cal P}_{j', \ell'+1}$. That is, for every $x \in {\cal P}_{j, \ell}$ there exists $u \in {\cal U}$ such that $x' = A^{j'} x + \sum^{j'} _{i=1} A^{i-1} B u \in {\cal P}_{j',\ell'}$.  
As illustrated in \rfig{add_edge}, we can then add a transition from $s_{j, \ell}$ to $s_{j', \ell'}$, since every $x \in {\cal P}_{j, \ell}$ can always enter ${\cal P}_{j', \ell'}$. 
Motivated by this fact, we define the following composite transition system: 
\begin{figure}[t]
   \centering
    \subfigure[If ${\cal P}_{j', \ell'}\subset {\cal P}_{j, \ell} \subseteq {\cal P}_{j', \ell'+1}$ (i.e., $a_{j', \ell'} < a_{j, \ell} \leq a_{j', \ell'+1}$) as illustrated in the above figure, then we add an edge from $s_{j, \ell}$ to $s_{j', \ell'}$.]
      {\includegraphics[width=6.7cm]{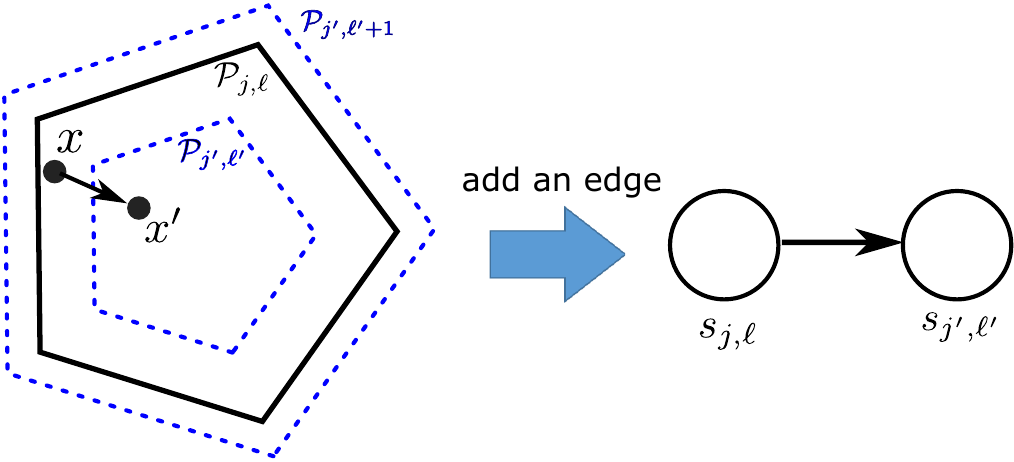}\label{add_edge}}
      \hspace{5pt}
    \subfigure[The figure shows an example of composite transition system $TS$ in \rdef{composition}. A terminal state $s_f$ and new transitions are added according to \rfig{add_edge}.]
      {\includegraphics[width=7.5cm]{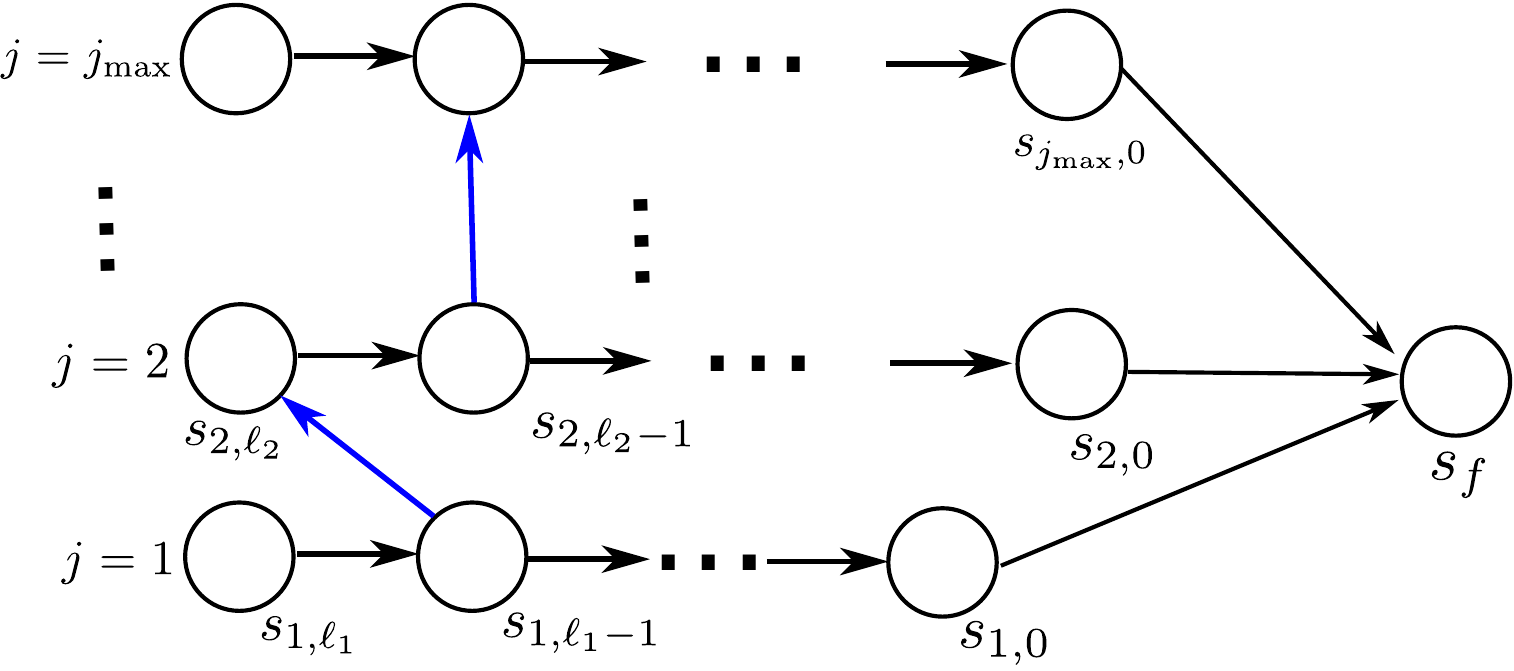}\label{graph_comp2}}
    \caption{Steps to construct a composite transition system}
 \end{figure}
\begin{mydef}\label{composition}
A composite transition system ${TS}$ is a tuple $TS = ( {\cal S}, \delta_{s}, {\cal R}_s, {\cal F}_s )$, where 
\begin{itemize}
\item ${\cal S} = \{s_f, {\cal S}_1, {\cal S}_2, \cdots, {\cal S}_{j_{\max}}\}$ is a set of symbolic states;
\item $\delta_{s} \subseteq {\cal S} \times  {\cal S}$ is a transition relation, where
\begin{enumerate}
\item $(s_{j, \ell}, s_{j, \ell-1}) \in \delta_s$ for all $j \in\{1, 2, \cdots, j_{\max}\}$, $\ell \in \{0, \cdots, \ell_j\}$;
\item $(s_{j, 0}, s_f) \in \delta_s$ for all $j \in\{1, 2, \cdots, j_{\max}\}$;
\item $(s_{j, \ell}, s_{j', \ell'}) \in \delta_s$ if 
\begin{equation}\label{acond}
a_{j', \ell'} < a_{j, \ell} \leq a_{j', \ell'+1},
\end{equation}
for $j, j'\in \{ 1, \cdots, j_{\max}\}$, $\ell \in \{ 0, \cdots, \ell_j\}$, $\ell' \in \{ 0, \cdots, \ell_{j'}\}$;
\end{enumerate}
\item ${\cal W}_s: {\delta_s} \rightarrow \mathbb{R}_+$ is the cost function associated with the transition $\delta_s$ given by; 
\begin{equation}\label{cost_composite}
{\cal W}_s ( \delta_s (s_{j, \ell} , s_{j', \ell'}) ) = p j'/(a_{j, \ell} - a_{j',\ell'}) + q /j'
\end{equation}
for $( s_{j, \ell} , s_{j', \ell'} ) \in \delta_s$, and ${\cal W}_s ( \delta_s (s_{j, 0} , s_f) ) = 0$ for all $j, j' \in \{1, 2, \cdots, j_{\max} \}$. 
\item ${\cal F}_s \subset {\cal S}$ is the terminal state given by ${\cal F}_s = \{ s_f \}$. 
\end{itemize}
\end{mydef}
The illustration of $TS$ is depicted in \rfig{graph_comp2}. In the composite transition system ${TS}$ defined above, an additional node $s_f$ and additional transitions among symbolic states are given. $s_f$ represents the terminal state and allows transitions $(s_{j,0}, s_f)\in \delta_s$ with the cost $0$ (${\cal W}_s (\delta_s(s_{j,0}, s_f)) =0$ ) for all $j\in \{1, 2, \cdots, j_{\max} \}$. Additional edges are given based on the motivation described above, and the corresponding costs are defined similarly to the $j$-step symbolic transition systems. 

Note that for every transition of two states $( s_{j, \ell} , s_{j', \ell'} ) \in \delta_s$, we have ${\cal P}_{j', \ell'} \subset {\cal P}_{j, \ell}$ since $a_{j', \ell'} < a_{j, \ell}$ (see \req{acond}). 
Moreover, since every cost associated with each transition is non-negative, one can always utilize well-known graph search algorithms such as Dijkstra algorithm, to find an optimal path to $s_f$ with the minimal total cost.
Finding such optimal path is a key idea to determine inter-event time steps to implement the self-triggered strategy in the next section. 


\section{Online implementation}
Having defined the composite transition system, let us now formulate the online implementations as the main result of this paper.

\subsection{Self-triggered strategy}
As mentioned in \rsec{strategy_sec}, the proposed control strategy consists of two parts; the first one is to steer the state to the target set ${\cal P}_0$ through the self-triggered strategy, and the second one is to stabilize to the origin inside ${\cal P}_0$ with the maximal possible inter-event time steps. In this subsection we first present the self-triggered strategy. In the following proposed self-triggered algorithm, 
let ${\pi}: {\cal S} \rightarrow {\cal X}$ be the mapping from a symbolic state $s_{j, \ell}$ in ${\cal S}$ to the corresponding set ${\cal P}_{j,\ell}$, i.e., $\pi ( s_{j, \ell} ) = {\cal P}_{j, \ell}$. \\

\noindent
\textbf{Algorithm 2 (Self-triggered strategy)}: 
\begin{enumerate}
\item (Initialization) : Set $k_0 = 0$.
\item For an update time $k_m$, $m\in \mathbb{N}$ and $x(k_m)$, the controller computes the pair $(j^*, \ell^*)$ by:
\begin{equation}\label{smallest_set}
\begin{aligned}
(j^*, \ell^* & ) = \underset{j, \ell}{\rm argmin}\ a_{j, \ell},\\
             {\rm s.t.}\ \ \ & x (k_m) \in {\cal P}_{j,\ell} (=a_{j, \ell} {\cal P}_0)
\end{aligned}
\end{equation}
for $j \in\{1,\cdots,j_{\max}\}, \ell \in \{0, \cdots,\ell_j \}$.
\item Find a finite path in $TS$:
\begin{equation}\label{optimal_path}
s_m (0), \ s_m (1), \cdots\ , s_m (d),
\end{equation}
where $s_m (0)=s_{j^*, \ell^*}$ and $s_m (d) = s_f$, such that the total cost $\sum^{d-1} _{j=0} {\cal W}_s (\delta_s(s_m(j ), s_m(j+1)))$ is minimized by applying, e.g., Dijkstra algorithm. 
\item 
Suppose that $s_m (1)$ is given by $s_m (1) = s_{j_p, \ell_p}$ for some $j_p \in\{1, \cdots, j_{\max} \}$, $\ell_p \in \{0, \cdots, \ell_{j_p} \}$, and the corresponding set being ${\cal P}_{j_p, \ell_p} = \pi (s_{j_p, \ell_p})$. Then, compute the control input $u^*$ to be applied by solving the following problem;
\begin{equation}\label{opt_controller}
u^* = \underset{u \in {\cal U}}{\rm argmin}\ \varepsilon
\end{equation}
subject to $\varepsilon \in [0, 1)$ and 
\begin{equation}
A^{{j}_p} x(k_m) + \sum^{{j_p}} _{i=1} A^{i-1} B u \in \varepsilon {\cal P}_{j_p, \ell_p}.
\end{equation}
\item The plant applies $u(k) = u^*$ for all $k \in [k_m, k_m + j_p )$. 
If $x(k_m + j_p ) \in {\cal P} _0$ then terminate the algorithm. Otherwise, set $k_{m+1} \leftarrow k_m + j_p $, $m \leftarrow m+1$, and go back to the step (2). \qedwhite
\end{enumerate}

Given the current state $x(k_m)$, $m\in \mathbb{N}$, Algorithm~2 starts by finding the smallest set containing $x(k_m)$, as shown in \req{smallest_set}. 
For a given pair $(j^*, \ell^*)$ obtained by \req{smallest_set}, we have $x(k_m) \in {\cal P}_{j^*, \ell^*}$ and $s_{j^*, \ell^*}$ is regarded as the current symbolic state in the composite transition system $TS$. The optimal path from $s_{j^*, \ell^*}$ to the terminal state $s_f$ is then found by applying a standard graph search algorithm (e.g., Dijkstra algorithm). 

Given the optimal path in \req{optimal_path}, the symbolic state $s_m (1)$, which is denoted as $s_{j_p, \ell_p}$ in Algorithm~2, indicates the next symbolic state that should be jumped from the current state $s_m (0) = s_{j^*, \ell^*}$. The corresponding set ${\cal P}_{j_p, \ell_p}$ indicates the next set to which the state trajectory should enter from ${\cal P}_{j^*, \ell^*}$. Since $(s_m (0), s_m(1)) = (s_{j^*, \ell^*}, s_{j_p, \ell_p}) \in \delta_s$, it holds that ${\cal P}_{j_p, \ell_p}\subset {\cal P}_{j^*, \ell^*}$ and there exists $u \in {\cal U}$ such that $A^{{j}_p} x(k_m) + \sum^{{j_p}} _{i=1} A^{i-1} B u \in {\cal P}_{j_p, \ell_p}$. 
Thus, we set $j_p$ as the inter-event time steps from $k_m$, i.e., the next update time $k_{m+1}$ is determined as $k_{m+1} = k_m + j_p$. The controller is then designed such that $x(k_{m+1}) \in {\cal P}_{j_p, \ell_p}$ by solving the linear programming problem as shown in \req{opt_controller}. 

The above procedure is repeated until the state achieves convergence to ${\cal P}_0$, and then switches to the second controller formulated in the next subsection. The following theorem is now concluded.
\begin{mythm}\label{stability_thm}
Consider the system \req{sys} subject to \req{constraint}, and that the proposed self-triggered strategy (Algorithm~2) is implemented. 
Then, any state trajectory starting from $x(k_0) \in {\cal P}_{\max}$, will enter the target set ${\cal P}_0$ in finite time. 
\end{mythm}
\begin{proof}
Suppose at the initial time we have $x(k_0) \in {\cal P}_{\max}$, 
and let $(j^* _0, \ell^* _0)$ be the pair obtained as a solution to \req{smallest_set} obtained at $k_0$, i.e., $x(k_0) \in {\cal P}_{j^* _0, \ell^* _0}\subseteq {\cal P}_{\max}$. Moreover, let \req{optimal_path} be the optimal path and $s_0 (1) = s_{j_p, \ell_p} $ for some $j_p \in\{1, \cdots, j_{\max} \}$, $\ell_p \in \{0, \cdots, \ell_{j_p} \}$. By solving \req{opt_controller}, we obtain a controller such that $x(k_1) \in  {\cal P}_{j_p, \ell_p} \subset {\cal P}_{j^* _0, \ell^* _0}$. 

Now, let $(j^* _1, \ell^* _1)$ be the pair obtained as the solution to \req{smallest_set} at the next update time $k_{1}$. Since $x ( k_{1} ) \in {\cal P}_{j_p, \ell_p}$ and ${\cal P}_{j^* _1, \ell^* _1}$ represents the \textit{smallest} set containing $x(k_1)$, we obtain $x(k_1) \in {\cal P}_{j^* _1, \ell^* _1} \subseteq  {\cal P}_{j_p, \ell_p}$.
Therefore, it holds that ${\cal P}_{j^* _1, \ell^* _1} \subseteq {\cal P}_{k_p, \ell_p} \subset  {\cal P}_{j^* _0, \ell^* _0}$ and thus ${\cal P}_{j^* _1, \ell^* _1}\subset {\cal P}_{j^* _0, \ell^* _0}$. We also have $x(k_0) \in  {\cal P}_{j^* _0, \ell^* _0}$ and $x(k_1) \in  {\cal P}_{j^* _1, \ell^* _1}$. By following the same procedure given above, it holds that: 
\begin{equation}\label{decreasingp}
{\cal P}_{j^* _0, \ell^* _0} \supset  {\cal P}_{j^* _1, \ell^* _1} \supset  {\cal P}_{j^* _2, \ell^* _2} \supset \cdots,
\end{equation}
and we have $x (k_m) \in {\cal P}_{j^* _m, \ell^* _m} $ for all $m\in \mathbb{N}$. This implies that the set containing $x(k_m)$ gets strictly smaller (i.e., closer to the origin) as the time sequence grows. This follows, that there exists $k_{M}, M \in \mathbb{N}$ such that $x ( k_{M} ) \in  {\cal P}_{j^* _M, \ell^* _M } = {\cal P}_0$; 
if this were not the case, \req{decreasingp} would be iterated for infinite times, which would mean that the number of generated sets in \req{generated_sets} would be infinite. 
This contradicts the fact that there exists only a finite number of generated sets as per \rrem{termination}. 
This completes the proof. 
\end{proof}

\subsection{Stabilizing inside the target set}
When stabilizing inside the target set ${\cal P}_0$, a periodic controller is designed such that the inter-event time steps can be maximized. 
Whether there exists a stabilizing, $j$-step constant controller inside ${\cal P}_0$ can be easily checked a-priori by solving the following problem; 
\begin{equation}\label{opt_controller2}
\varepsilon^* = \underset{u_1, \cdots, u_N \in {\cal U}}{\rm min}\ \varepsilon
\end{equation}
subject to $\varepsilon \in [0, 1)$ and 
\begin{equation}
A^{j} v_n + \sum^{j} _{i=1} A^{i-1} B u_n \in \varepsilon {\cal P}_{0},\ \ \forall n\in \{1, \cdots, N\}. 
\end{equation}
If \req{opt_controller2} has a solution for $j\in\{1, \cdots, j_{\max} \}$, then it is easily shown similarly to \rlem{stability_lem}, that there exists a stabilizing, $j$-step constant control law $g_j : {\cal U} \rightarrow {\cal X}$ such that the gauge function $\Psi_{{\cal P}_0}$ is strictly decreasing, i.e., $\Psi_{{\cal P}_0} (A^{j} x + \sum^{j} _{i=1} A^{i-1} B g_j (x)) \leq  \varepsilon^* \Psi_{{\cal P}_0}  (x)$. Note that \req{opt_controller2} always has a solution for $j=1$, since ${\cal P}_0$ is assigned as the $\lambda$-contractive set (see \rsec{targ_assign_sec}). 

By checking the feasibility of \req{opt_controller2} for each $j \in \{1, \cdots, j_{\max} \}$, the maximal inter-event time steps to guarantee stability in ${\cal P}_0$ is obtained as:
\begin{equation}
\hat{j} = {\rm max} \left \{ j\in\{1,\cdots, j_{\max}\} : {\req{opt_controller2}\ {\rm is\ feasible}} \right \}. 
\end{equation}
Having obtained the maximal inter-event time steps $\hat{j}$ and by using a similar procedure to \req{opt_controller}, the following state feedback controller is designed for $x(k_m) \in {\cal P}_0$: 
\begin{equation}\label{stabilize_target}
u^* = \underset{u \in {\cal U}}{\rm argmin}\ \varepsilon
\end{equation}
subject to $\varepsilon \in [0,  1)$ and 
\begin{equation}
A^{\hat{j}} x(k_m) + \sum^{\hat{j}} _{i=1} A^{i-1} B u \in \varepsilon {\cal P}_{0}. 
\end{equation}
Note that instead of solving \req{stabilize_target}, other control schemes may be used to design a state feedback controller, such as linear variable structure control, see e.g., \cite{blanchini1994a}. 
\section{Illustrative example}
Consider again the example of the double integrator system described in \rsec{enlarge_sec} (Example 1). The resulting composite transition system $TS$ consists of $542$ symbolic states and $1052$ transitions. 
For simplicity reasons, only Algorithm~2 (self-triggered strategy) is implemented to verify \rthm{stability_thm}. 
\rfig{state_traj} plots the state trajectory by applying Algorithm~2 from the initial point $x (k_0) = [0\ ; -4]$ with the weight $(p, q) = (1,1)$ in the cost function \req{cost_composite}, and \rfig{triggering} indicates the corresponding control input. The average inter-event time step is obtained by $11.0$ (step). 
From these figures, it is shown that the state trajectory enters the target set by aperiodically executing control tasks according to Algorithm~2. 

To analyze the effect of weights $(p, q)$, we also implement Algorithm~2 with the two cases $(p, q) = (1,1), (1, 10)$ and then evaluate the control performance and the inter-event time steps. In both cases, Algotirhm~2 is conducted for $1000$ times, where for each time the initial state is randomly selected satisfying $x(k_0) \in {\cal P}_{\max} \backslash {\cal P}_0$. In \rtab{weight_result}, we illustrate the resulting average inter-event time steps and the average time steps required to converge to the target set. From the table, the average inter-event time steps become larger for the case $(1, 10)$ than for the case $(1,1)$. This is due to the fact from \req{cost_composite}, that the cost for the inter-event time steps is more penalized as $q$ is largely selected. On the other hand, the convergence time steps for the case $(1,10)$ become longer than for the case $(1,1)$, which implies that a better control performance is achieved when $(1,1)$. Therefore, it is shown that there exists a trade-off between achieving the control performance and the inter-event time steps, and this trade-off can be regulated by appropriately tuning these parameters.

\begin{figure}[t]
   \centering
    \subfigure[State trajectory (blue solid line) from the initial state ${x(k_0)=[0; -4]}$ (red cross mark) and the domain of attraction (black solid line).]
      {\includegraphics[width=7.5cm]{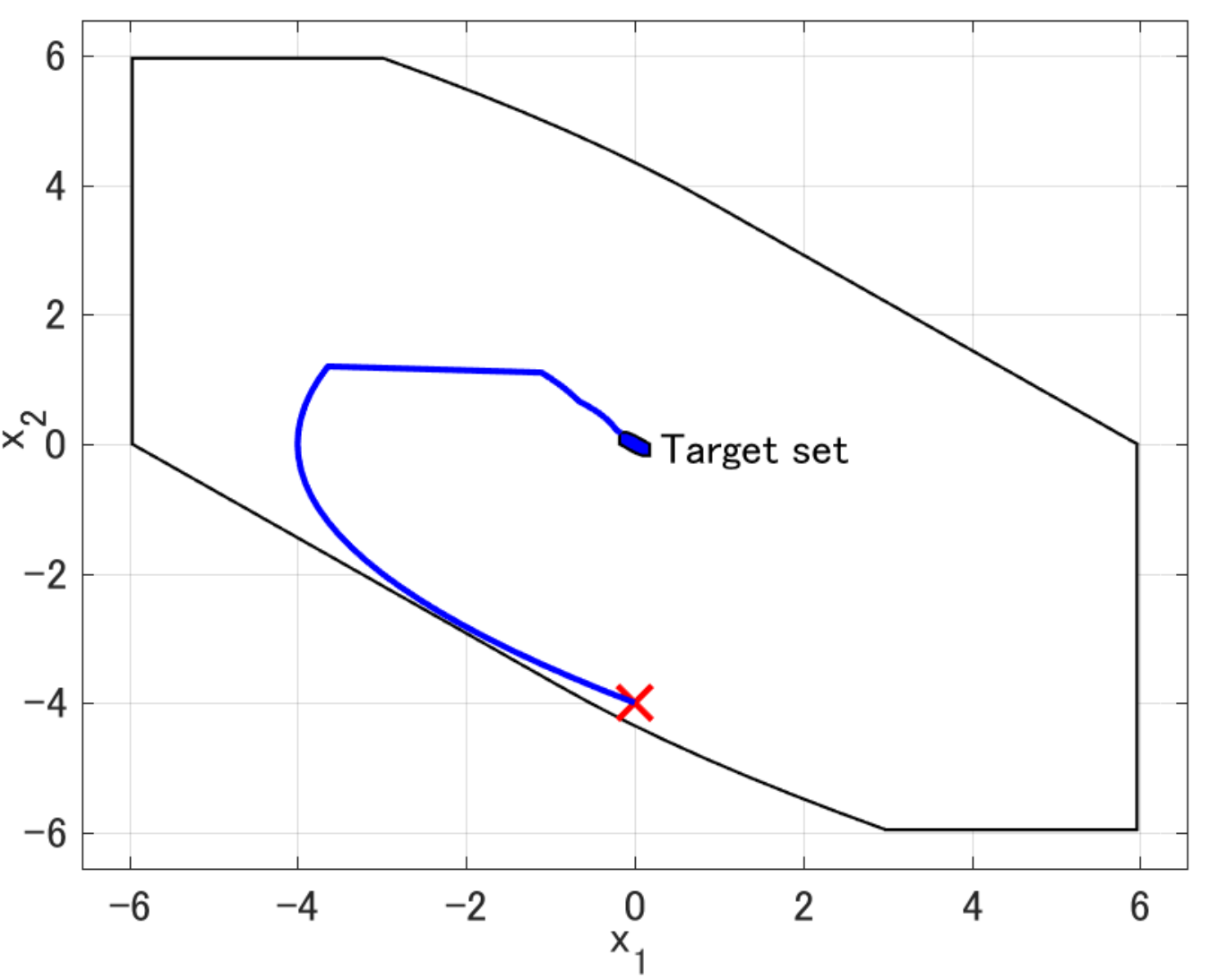}\label{state_traj}}
    \subfigure[Applied control input $u(k)$.]
      {\includegraphics[width=7.5cm]{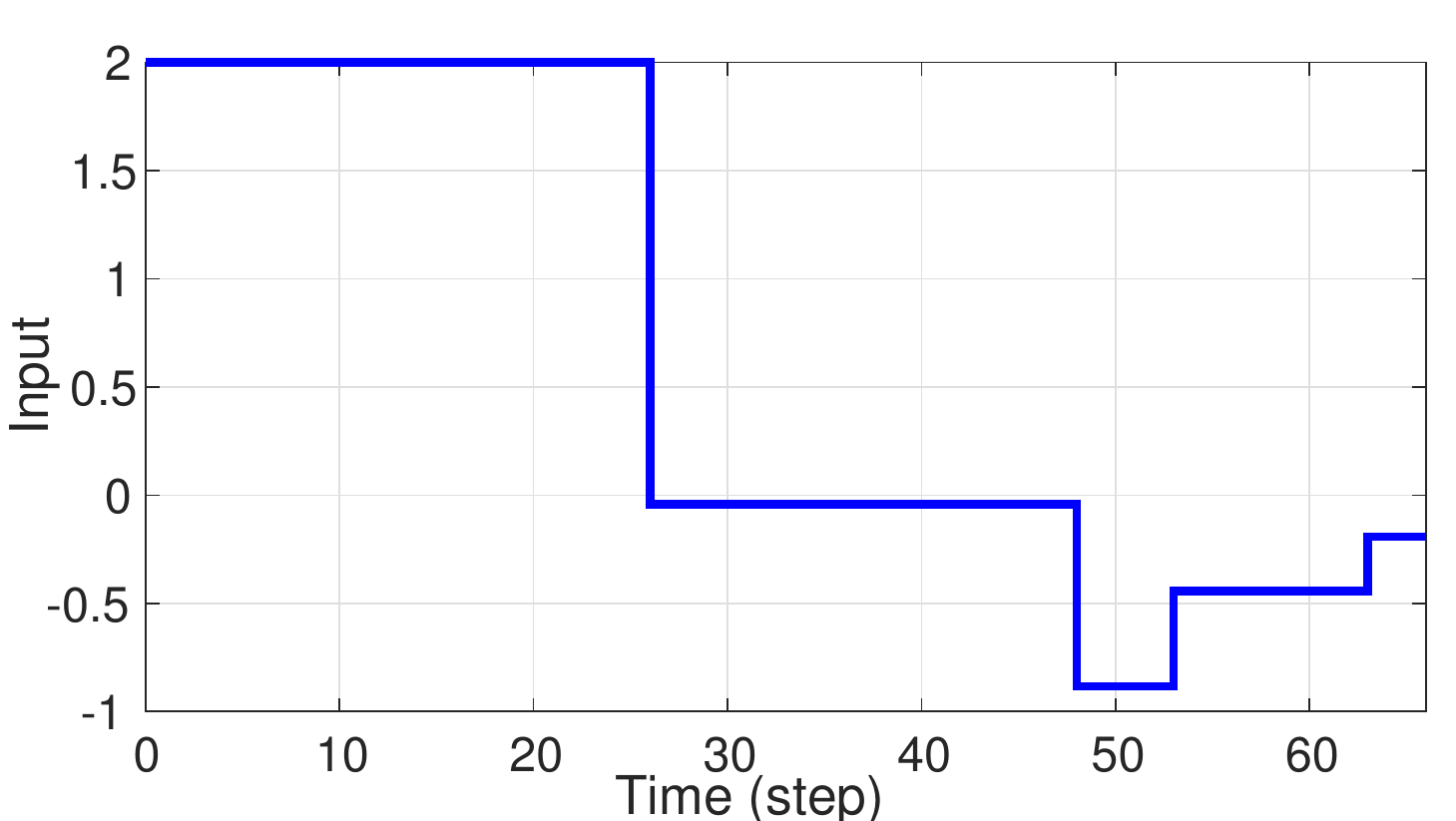}\label{triggering}}
    \vspace{-5pt}
    \caption{State trajectories and the corresponding control input obtained by Algorithm~2.}
 \end{figure}
\begin{table}[h]
\begin{center}
\caption{Average inter-event time steps and convergence time steps}\label{weight_result}
\begin{tabular}{|c||c|c|} \hline
$(p,q)$   & Inter-event (steps)  & Convergence (steps) \\ \hline
$(1,1)$      & 14.7       &  55.6  \\ \hline
$(1, 10)$  &   23.6       &  78.1 \\ \hline
\end{tabular}
\end{center}
\end{table}

\balance
\section{Conclusions and Future work}
In this paper, we have proposed a self-triggered strategy for constrained discrete time systems. The proposed scheme is to construct a collection of polyhedral contractive sets through vertex operations, and then translate them into the corresponding {transition systems}, consisting of symbolic states and transitions to represent the system's behavior. 
In the self-triggered strategy, inter-event time steps are determined by applying a shortest path algorithm for each update time. The proposed scheme was also validated by an illustrative example. 

In the off-line phase, it is required in Algorithm~1 that a collection of polyhedral sets needs to be computed for all $j\in \{1, \cdots, j_{\max}\}$. Thus, the computational complexity may become high as the state dimension $n$ increases. Thus, it is important in our future work, to analyse the computational complexity of the proposed scheme. Also, future work involves extending the proposed framework to the case of random packet dropouts or network delays.


\appendix 
(\textit{Proof of \rlem{convergence}}):
Suppose that $x \in {\cal P}_{j, \ell}$ with $j\in \{1, \cdots, j_{\max}\}$, $\ell \in \{ 1, \cdots, \ell_j \}$. Since $x \in {\cal P}_{j, \ell}$, there exist $\mu_n \in [0, 1]$, $n\in \{1, \cdots, N \}$ such that $x = \sum^N _{n=1} \mu_n (a_{j, \ell } v_n)$ with  $\sum^N _{n=1} \mu_n = 1$. 
Furthermore, from \req{constx} in Algorithm~1, there exist $u_n \in {\cal U}$, $n\in \{1, \cdots, N \}$ such that $A^j a_{j, \ell} v_n + \sum^{j-1} _{i=1} A^{i-1} B u_n \in {\cal P}_{j,\ell-1}$. Set $u = \sum^N _{n=1} \mu_n u_n \in {\cal U}$. Then we obtain $x' = A^j x + \sum^{j-1} _{i=1} A^{i-1} B u = \sum^N _{n=1} \mu_n \left ( A^j a_{j, \ell} v_n + \sum^{j-1} _{i=1} A^{i-1} B u_n \right ) \in {\cal P}_{j, \ell-1}$. Therefore, there exists $u \in {\cal U}$ such that $x' \in {\cal P}_{j, \ell-1}$. This completes the proof. 
\end{document}